\documentclass{amsart}
\usepackage{amsmath}
\usepackage{amssymb}
\usepackage{color}
\input xy
\xyoption{all}
\newtheorem{Lemma}{Lemma}[section]
\newtheorem{remark}[Lemma]{Remark}

\newtheorem{theorem}[Lemma]{Theorem}
\newtheorem{lemma}[Lemma]{Lemma}
\newtheorem{proposition}[Lemma]{Proposition}
\newtheorem{corollary}[Lemma]{Corollary}

\newtheorem{example}[Lemma]{Example}

\newenvironment{Proof}{{\sc Proof.}\ }{~$\square$\vspace{0.2truecm}}

\newcommand{\End}{\operatorname{End}}

\newcommand{\cmat}{\left(\begin{array}}
\newcommand{\fmat}{\end{array}\right)}

\begin{document}
   \title{$V$-rings versus $\Sigma$-$V$ Rings}
 \author{Bijan Davvaz}
 \address{Department of Mathematics, Yazd University, Yazd, Iran}
 \email{davvaz@yazd.ac.ir}
 \author{Zahra Nazemian}
 \address{Department of Mathematics, Yazd University, Yazd, Iran}
 \email{z$_{-}$nazemian@yahoo.com}
  \author{Ashish K. Srivastava}
\address{Department of Mathematics and Computer Science, St. Louis University, St. Louis, MO-63103, USA}
 \email{asrivas3@slu.edu}
      \begin{abstract} 
      This paper studies similarities and differences between the classes of rings over which each simple module is injective and rings over which each simple module is $\Sigma$-injective. The rings in the former class are called $V$-rings and the rings in the latter class are called $\Sigma$-$V$ rings. We have obtained analogues of various well-known results about $V$-rings for $\Sigma$-$V$ rings. Motivated by a conjecture of Kaplansky, Fisher asked if a prime right $V$-ring is right primitive. Although a counter-example to Kaplansky's conjecture was constructed long ago but Fisher's question is still open. In this paper we show that for a right $\Sigma$-$V$ ring, the notions of prime and primitive are equivalent. Also, we show that an exchange $\Sigma$-$V$ ring is left-right symmetric and moreover, it is von Neumann regular.  
      %as a characterization of right $V$-rings, it is shown that a ring $R$ is a right $V$-ring if and only if every right $R$-module is cogenerated by the class of iso-simple right $R$-modules.   
       \end{abstract}
      
    \maketitle

\section{Introduction}

\noindent A ring $R$ is called a right $V$-ring if each simple right $R$-module is injective \cite{Vil}. These rings were first studied by Villamayor and so in his honor these are called $V$-rings. It is a well-known result due to Kaplansky that a commutative ring is von Neumann regular if and only if it is a $V$-ring. However, in the case of noncommutative setting, the classes of von Neumann regular rings and $V$-rings are quite independent. A generalization of the notion of $V$-rings was introduced in \cite{GV} and studied subsequently in \cite{Baccellaproc} and \cite{Ashish} where a ring $R$ is called a right $\Sigma$-$V$ ring if each simple right $R$-module is $\Sigma$-injective. Recall that a module $M$ is called $\Sigma$-injective if $M^{(\alpha)}$ is injective for any cardinal $\alpha$. Clearly, a $\Sigma$-$V$ ring is a $V$-ring, however, there are examples of $V$-rings that are not $\Sigma$-$V$ rings. In \cite{Ashish} it is proved that a right $\Sigma$-$V$ ring is directly-finite and an example is given of a right $V$-ring that is not directly-finite. Baccella \cite{Baccellaproc} had shown that a von Neumann regular ring $R$ is a right $\Sigma$-$V$ ring if and only if $R/P$ is simple artinian for each prime ideal $P$ of $R$ and consequently, it follows that for a von Neumann regular ring the notion of $\Sigma$-$V$ ring is left-right symmetric. Examples are known of von Neumann regular right $V$-rings that are not left $V$-rings.       

In this paper, we continue investigation of right $\Sigma$-$V$ rings and explore further similarities and differences between the classes of $V$-rings and $\Sigma$-$V$ rings. Among other things we prove that for a $\Sigma$-$V$ ring, the notions of prime and primitive are equivalent thus giving a partial answer to a question of Fisher \cite{Fisher} which asks if a prime right $V$-ring is right primitive. We show that for class of exchange rings, the notion of $\Sigma$-$V$ rings is left-right symmetric. We also show that an exchange right $\Sigma$-$V$ ring has bounded index of nilpotence which gives a positive answer to \cite[Problem 15]{Ashish}.

Throughout this paper all rings are associative with identity element and modules are unital right modules unless stated otherwise. For any term not defined here, the reader is referred to \cite{Anderson}, \cite{Lam1} and \cite{lam2}.    

\bigskip

\section{Some characterizations of $\Sigma$-$V$ rings}

\noindent We begin with some useful equivalent characterizations for a ring to be a right $\Sigma$-$V$ ring. It is known that a ring $R$ is a right $V$-ring if and only if  each right ideal $I$ of $R$ is an idempotent, that is, $I^2=I$ and every primitive factor ring of $R$ is a right $V$-ring. We have the following analogue for right $\Sigma$-$V$ rings.

\begin{lemma} \label{right primitive factor}
For a ring $R$, the following statements are equivalent:
\begin{enumerate}
\item $R$ is a right $\Sigma$-$V$ ring. 
\item Each right ideal of $R$ is an idempotent and each primitive factor ring of $R$ is a right $\Sigma$-$V$ ring.
\end{enumerate}
\end{lemma}

\begin{proof}
$(1) \Rightarrow (2)$. Let $R$ be a right $\Sigma$-$V$ ring. Clearly then every factor ring of $R$ is a right $\Sigma$-$V$ ring. 
Since a right $\Sigma$-$V$ ring is, in particular, a right $V$-ring, every right ideal of $R$ is an idempotent (see, for example \cite [Theorem 6.2] {JAT}).  \\

$(2) \Rightarrow (1)$. Let $S$ be a simple right $R$-module. To prove that $R$ is a right $\Sigma$-$V$ ring, we need to show that $S$ is $\Sigma$-injective. Let $I= ann_r(S)$. Then $R/I$ is a right primitive ring. 
Let $M = \oplus _{i = 1}^{\infty} S_i$ where each $S_i=S$. It is clear that $ann_r(M) = I$. We claim  that $ann_r(E(M)) = I$. Assume to the contrary that 
$I \neq ann_r(E(M))$. Then there is an element $x \in E(M)$ such that $xI \neq 0$. Let $a \in I$ such that $xa \neq 0$, then there exists 
  $r \in R$ such that $xar \in M$ is nonzero.  Set $K = arR$. Since $xK \leq M$ and $K$ is idempotent, we have $xK \leq MK \leq MI = 0$.
  This yields a contradiction. Hence $I = ann_r(E(M))$. So $E(M)$ is an $R/I$-module. Since, by assumption $R/I$ is a right $\Sigma$-$V$ ring, it follows that $E(M) = M$. Hence $S$ is $\Sigma$-injective. This shows that $R$ is a right $\Sigma$-$V$ ring. 
\end{proof}

\noindent Recall that a ring $R$ is called right {\it $q.f.d.$ $($quotient finite-dimensional$)$ relative to a module $M$} if no cyclic right $R$-module contains an infinite direct sum of modules isomorphic to submodules of $M$. In \cite{Ashish}, it is shown that if $R$ is a right $\Sigma$-$V$ ring, then $R$ is q.f.d. relative to every simple right $R$-module. We extend it in the next lemma.

\begin{lemma} \label{qfd}
For a ring $R$, the following statements are equivalent:
\begin{enumerate}
\item $R$ is a right $\Sigma$-$V$ ring. 
\item $R$ is a right $V$-ring and $R$ is q.f.d. relative to every simple right $R$-module.
\end{enumerate}
\end{lemma}

\begin{proof}
$(2) \Rightarrow (1)$. Let $R$ be a right $V$-ring which is q.f.d. relative to every simple right $R$-module. Let $S$ be a simple right module
 and $M = \oplus _{i = 1}^ {\infty } S_i$, where $S_i=S$. We claim that $E(M)$ is semisimple. Since $E(M)$ is a sum of cyclic 
submodules, it is enough to see that cyclic submodules of $E(M)$ are semisimple. Let $xR$ be a cyclic submodule of $E(M)$ and so it has an essential socle. Because of our 
assumption its socle is finitely generated and so $xR$ is semisimple. Thus $E(M)$ is semisimple and hence it follows that $E(M)=M$ and consequently, $S$ is $\Sigma$-injective. This shows that $R$ is a right $\Sigma$-$V$ ring.

$(1) \Rightarrow (2)$.  This is Lemma 1 in \cite{Ashish} and it follows from simple observation that if a right $R$-module $M$ is $\Sigma$-injective, then $R$ is right $q.f.d.$ relative to $M$.
\end{proof}

As a consequence, we have the following.

\begin{corollary}
For a right $\Sigma$-$V$ ring, the following are equivalent:
\begin{enumerate}
\item $R$ is right noetherian.
\item No cyclic right $R$-module contains an infinite direct sum of non-isomorphic simple modules.
\item $R$ is right q.f.d. 
\end{enumerate}
\end{corollary}

Recall that a module $M$ is called {\it directly-finite $($or dedekind-finite$)$}, if it is not isomorphic to any of its proper direct summand. 

\begin{proposition}
If $R$ is a right $\Sigma$-$V$ ring, then every finitely generated right $R$-module is directly-finite.
\end{proposition}
\begin{Proof}
Let $M$ be a finitely generated module over $R$ where $R$ is a right $\Sigma$-$V$ ring. Assume to the contrary that $M$ is not directly-finite. Then $M = M_1 \oplus A_1$, $M_1 = M_2 \oplus A_2,$ and so on, where for each $n,$ $M_n \cong M,$ and
 $A_1 \cong A_2 \cong A_3 \cong \cdots$. Then we can find, in each $A_n,$ a maximal submodule $B_n$ such
that $A_n / B_n$ are all isomorphic to the same simple module. Then, by assumption, $\oplus _n A_ n / B_ n$
 is injective and hence it splits in $M / (\oplus _ n B_n)$, a contradiction, since
$M / (\oplus _ n B_n)$ is finitely generated. Therefore $M$ is directly-finite.
\end{Proof}

We do not know if the converse of the above result is true or not in case $R$ is a right $V$-ring. We ask

\bigskip

\noindent {\bf Question 1.} Let $R$ be a right $V$-ring such that each finitely generated right $R$-module is directly-finite. Must $R$ be a right $\Sigma$-$V$ ring?   \\

We give below an example of a right $V$-ring $R$ whose cyclic right modules are directly-finite but $R$ is not a right $\Sigma$-$V$ ring. 

\begin{example}\label{2.5}
 {\rm Let $F$ be a field and $V$, a vector space of infinite dimension over $F$. Let $S=\End_F(V)$,
 $J$ be the socle of $S_S$ and $R = F + J$. It may be checked that thing $R$ is a right $V$-ring but not a left $V$-ring. Since this ring is unit regular (see \cite [Example 5.15]{Goo79}) and not a left $V$-ring, it is clear that $R$ cannot be a right $\Sigma$-$V$ ring because if it was a right $\Sigma$-$V$ ring then being a von Neumann regular it would be a left $\Sigma$-$V$ ring and hence a left $V$-ring. Now we claim that every
 cyclic right $R$-module  is directly finite. Let $R/I$, where $I$ is a right ideal of $R$, be a cyclic right module. 
 If $ I \nsubseteq J$, then $I +J = R$, so $R/I$ is semisimple.
  Since it is also cyclic, it is directly finite. Now assume $I \subseteq J$. Note that $\End_R(R/I) = S/I$, where $S = \{r \in  R \mid  rI \subseteq I \}$ is the
idealizer of $I$. Suppose $f, g \in  S/I$ with $fg = 1.$ Then $f = s + I$ and $g = t + I$ for some
$s, t \in  S$ with $ st - 1 \in  I.$  Now $(s + J)(t + J) = 1 + J$ in $R/J = F \centerdot 1$ (because $I \subseteq J$), so
there exists $\alpha \neq 0 \in  F$  and $a, b \in  J$ such that $s = \alpha + a$ and $ t = \alpha ^{ -1 } + b.$ 
Since we may replace $f$ and $g$ by  $\alpha ^{-1}f$ and $\alpha g,$ there is no loss of generality in assuming that $\alpha = 1.$ So,
we now have $s = 1 + a$ and $t = 1 + b.$
Note that $s, t \in S$ implies $a, b \in  S.$ Also, $st - 1 \in I$ implies $a + b + ab \in  I.$
Choose a right ideal $I''$ of $R$ such that $J = I \oplus I''$, and write $a = a'+a''$ and $b = b' +b''$
for some $a', b' \in  I$ and $a'', b'' \in  I''$. Note that because $a', b' \in  I \subseteq S, $ we have $a'', b''  \in  S.$
Since $s + I = 1 + a'' + I$ and $t + I = 1 + b'' + I,$ we may replace $a$ and $b$ by $a''$ and $b''$.
Thus, there is no loss of generality in assuming that $a, b \in  I''.$
At this point, $a+b+ab \in I \cap  I'' = 0,$ which implies $st = 1.$ Since $R $ is unit-regular and
hence directly finite, $ts = 1.$ Therefore $gf = 1,$ proving that $S/I$ is directly finite, whence
the module $R/I$ is directly finite.}
\end{example}

The Jacobson radical rad$(M)$ of a module $M$ is defined as the intersection of all maximal submodules of $M$. It is not difficult to see 
that a right $R$-module $M$ has zero Jacobson radical if and only if $M$ can be embedded in a product of simple right modules, i.e., $M$ 
is cogenerated by the class of simple modules.  Following \cite{paper}, recall that
 a nonzero module is called iso-simple if it is isomorphic to all of its nonzero submodules. Simple modules and principal right ideal domains
as right module  are obvious examples of iso-simpe right modules. 

\begin{theorem} For a ring $R$ the following are equivalent:\\
{\rm (1)} Every right $R$-module can be embedded in a product of right iso-simple modules.\\
{\rm (2)} Every right $R$-module has zero Jacobson radical.\\
{\rm (3)} $R$ is a right $V$-ring.  
   \end{theorem}
    \begin{Proof}
  $ (1) \Rightarrow (2)$. It is enough to show that iso-simple right modules have zero Jacobson radical. Let $I$ be an iso-simple right module such that rad$(I)$ is nonzero. Thus $I$ is not simple. Note that $I$ and so rad$(I)$ is  noetherian. We can take a maximal submodule of rad$(I)$, say $M$. 
  To get to a contradiction, we show that for every iso-simple right module $T$ and every $h : I/M \longrightarrow T$, rad$(I) /M \leq $ ker$(h)$. Let $T$ be an iso-simple module and $h : I/M \longrightarrow T$ be a nonzero map. If $T$ is simple, then ker$(h)$ is 
  maximal in $I/M$ and so  rad$(I) /M \leq $ ker$(h)$. Now assume that $T$ is not simple. In this case since $I/M$ has nonzero socle,
  ker$(h)$ is nonzero. Let $K \geq M$ be a submodule of $I$ such that 
   ker$(h) = K/M$. If rad$(I) /M$ is not a submodule of $K/M$, the restriction $h$ to $($rad$(I) +  K )/M$, implies a nonzero map from  
  $($rad$(I) +  K )/M$ to $T$ whose kernel is $K/M$. Thus $T \cong $ rad$(I)/ ($rad $(I) \cap K)$.
  Since $M \leq  ($rad$(I) \cap K)$ and $M$ is maximal in rad$(M)$. We have $M = ($rad$(I) \cap K)$. This shows
  $T$ is simple which is a contradiction. Thus rad$(I) /M$ is a submodule of ker$(h)$.\\
$ (2) \Rightarrow (1)$ is clear and $ (2) \Leftrightarrow  (3)$ follows by \cite [ Theorem 3.75 ] {lam2}.
   \end{Proof}

\begin{remark} \label{Cozzens}
{\rm It is not difficult to see that if an iso-simple module is injective, then it is a simple module. Thus if all iso-simple right $R$-modules are 
injective, then $R$ is right $V$-ring. The converse is not true, in general. If a ring $R$ is a right $V$-ring, then the iso-simple right $R$-modules are not necessarily injective. For example,   if $R$ is a principal right
 ideal $V$-domain such that $R$ is not a division ring then $R_R$ is not injective. One can find examples of such domains constructed by Cozzens in \cite {Cozzens}.}
\end{remark}

In \cite{paper}, a module $M$ is called {\em iso-artinian} if for every descending chain $ M_1 \geq M_2 \geq \cdots$ of 
submodules of $M$, there exists an $n\ge1$ such that $M_n \cong M_i$ for all $i \geq n$. Every iso-artinian module contains an essential 
submodules which is a direct sum of iso-simple modules (see \cite [Theorem 2.9] {paper}).
 
\begin{proposition}
For a right $\Sigma$-$V$ ring $R$, the following are equivalent:\\
{\rm (1)} Every iso-simple right $R$-module is injective.\\
{\rm (2)} Every iso-artinian right $R$-module is injective and semisimple.
\end{proposition}
   \begin{Proof}
   $ (1) \Rightarrow (2)$. Following  \cite [Theorem 2.9] {paper},
   Every iso-artinian module contains an essential 
submodules which is a direct sum of iso-simple modules.  
 So if $M$ is an iso-artinian module, it has an essential socle which must be iso-artinian. 
 Thus soc$(M)$ is a direct sum of homogeneous semisimple modules an so it is injective.
 Then $M = $soc$(M)$ is a semisimple injective module.  \\
   $ (2) \Rightarrow (1)$ is clear.     
   \end{Proof}
   
   \begin{example}
   {\rm Note that a ring whose iso-simple right modules are injective need not be a right $\Sigma$-$V$ ring.
     Take the ring in Example \ref{2.5}. Then $R$ is a von Neumann regular right $V$-ring which is not a right $\Sigma$-$V$ ring. Also, it may checked that $R$ is both left and right semi-artinian. So, each noetherian right $R$-module is injective and a finite direct sum of simple modules.   }
   \end{example}

\bigskip

\section{Kaplansky's Conjecture and a question of Fisher}

\noindent Kaplansky conjectured that a prime von Neumann regular is primitive \cite{Kap}. Fisher and Snider verified the conjecture for countable rings \cite{FS} and Goodearl verified the conjecture for right self-injective rings \cite{G}. Recently, Ara and Bell have shown that a prime countable-dimensional von Neumann regular algebra over any field is primitive \cite{AB}. However, in general, Kaplansky's conjecture is not true and a counter-example was constructed by Domanov \cite{Domanov}. Abrams, Bell and Rangaswamy have constructed an example of a prime non-primitive von Neumann regular Leavitt path algebra in \cite{ABR}. Inspired by Kaplansky, Fisher raised the question whether a prime right $V$-ring is right primitive \cite{Fisher}. This question is still open. In our next theorem we will show that for a right $\Sigma$-$V$ ring, the notions of prime and primitive are equivalent.       

Before presenting the proof of this theorem, we recall basics of the structure theory of von Neumann regular right self-injective rings as our proof depends heavily on this structure theory. The theory of types was first proposed by Murray and von Neumann \cite{MV} but it was developed as a tool for classification by Kaplansky in \cite{Kap1} for a certain class of rings of operators which are called Baer rings. Since von Neumann regular right self-injective rings are Baer rings, Kaplansky's theory applies to von Neumann regular right self-injective rings. A von Neumann regular right self-injective ring is said to be of type $I$ provided it contains a faithful abelian idempotent and it is said to be of type $II$ provided $R$ contains a faithful directly-finite idempotent but no nonzero abelian idempotents. A von Neumann regular right self-injective ring is of type $III$ if it contains no nonzero directly-finite idempotents. Moreover, a von Neumann regular right self-injective ring is said to be of type $I_{f}$ (resp., $I_{\infty}$) if $R$ is of type $I$ and is directly-finite (resp., purely-infinite) and it is said to be of type $II_{f}$ (resp., $II_{\infty}$) if $R$ is of type $II$ and is directly-finite (resp., purely-infinite). It is well known that any von Neumann regular right self-injective ring $R$ can be decomposed as a product $R=R_1 \times R_2 \times R_3 \times R_4 \times R_5$ where $R_1$ is of type $I_f$, $R_2$ is of type $I_\infty$, $R_3$ is of type $II_f$, $R_4$ is of type $II_\infty$, and $R_5$ is of type $III$ (see \cite{Goo79}, pp. 111-115).

\begin{theorem} \label{equivalent}
For a right $\Sigma$-$V$-ring $R$, the following are equivalent:
\begin{enumerate}
\item $R$ is right primitive.
\item $R$ is prime. 
\item $R$ is left primitive.
\end{enumerate}
\end{theorem}

\begin{Proof}
$(2) \Rightarrow (3)$. Let $R$ be a prime right $\Sigma$-$V$ ring. Following \cite [Lemma 4.3] {On flat factor}, $R$ is right nonsingular. Thus the maximal right quotient ring $Q=Q_{\max}^{r}(R)$ is a prime von Neumann regular right self-injective ring. We know that $R$ is directly-finite being a right $\Sigma$-$V$ ring and consequently, the ring $Q$ is also directly-finite. 

Now, by the type theory of von Neumann regular right self-injective rings, $Q$ is a direct product of rings of type $I_f$ and type $II_f$ but since $Q$ is a prime ring, it follows that either $Q$ is of type $I_{f}$ or $Q$ is of type $II_{f}$ (see \cite{Goo79}, Theorem 10.22). We claim that $Q$ must be of type $I_{f}$. 

Assume to the contrary that $Q$ is of type $II_{f}$. Then, by (\cite{Goo79}, Proposition 10.28), there exists an idempotent $e\in Q$ such that $Q_{Q}\cong 3(eQ)$. Therefore $Q=e_{1}Q\oplus e_{2}Q\oplus e_{3}Q$ where $e_{1}, e_{2}, e_{3}\in Q$ are nonzero orthogonal idempotents such that their sum is the identity of the ring $Q$ and $e_{i}Q\cong e_{j}Q\cong eQ$ for all $1\leq i,j\leq 3$. Therefore there exist nonzero cyclic submodules $C_{2i}\subseteq e_{i}Q\cap R,$ $i=1,2,$ such that $C_{21}\cong C_{22}$. 

Now, as $eQe=\End(eQ)$ is also of type $II_{f}$, there exist nonzero orthogonal idempotents $f_{1},f_{2},f_{3},f_{4}\in eQe$ such that $f_{i}(eQe)\cong $ $f_{j}(eQe)$ for all $1\leq i,j\leq 4$. Hence $f_{i}(eQe)\cong f_{j}(eQe)$ for all $i,j$. By (\cite{Lam1}, Proposition 21.20), there exist $a\in f_{i}(eQe)f_{j}$ and $b\in f_{j}(eQe)f_{i}$ such that $f_{i}=ab$ and $f_{j}=ba$. 

Define $\varphi: f_iQ\rightarrow f_jQ$ by $\varphi(f_iq)=bf_iq$ for each $q\in Q$. Clearly, $\varphi$ is an isomorphism. Thus, for all $i,j$, we have $f_{i}Q\cong f_{j}Q$. Furthermore, there exist nonzero cyclic submodules $C_{3i}\subseteq f_{i}Q\cap R,$ $i=1,2,3$ such that $C_{3i}\cong C_{3j}$ for all $1\leq i,j\leq 3$. 

Continuing in this fashion, we construct an independent family $\{C_{ij}$ : $i=2,3, \ldots;$ $1\leq j\leq i\}$ of nonzero cyclic submodules of $R$ such that $C_{ij}\cong C_{ik}$ for all $1\leq j,k\leq i$; $i=2,3, \ldots$. Therefore there exist maximal submodules $M_{ij}$ of $C_{ij},1\leq j\leq i$ $ ;i=2,3, \ldots$ such that $C_{ij}/M_{ij}\cong C_{ik}/M_{ik}$ for all $i,j,k.$ 

Setting $M=\oplus _{i,j}M_{ij}$, and $S_{i}=C_{i1}/M_{i1}$, we get that the cyclic right module $R/M$ contains an infinite direct sum of modules each isomorphic to $S_{i}$, a contradiction to Lemma \ref{qfd}. Therefore $Q$ is of type $I_{f}$. 

Hence $Q$ is a product of matrix rings over abelian regular rings but again since $Q$ is a prime ring, it turns out that $Q$ is a matrix ring over an abelian regular ring, say $S$ (see \cite{Goo79}, Theorem 10.24). This means that $S$ is a prime abelian regular ring. Since $S$ being a prime ring has no nontrivial central idempotent and each idempotent in an abelian regular ring is central, we conclude that $S$ is a von Neumann regular ring with no nontrivial idempotent and hence a division ring. Thus $Q$ is a simple artinian ring and consequently, $R$ is a simple right Goldie ring. Hence $R$ is right and left primitive.\\

$(1) \Rightarrow (2)$, $(3) \Rightarrow (1)$ are clear. 
\end{Proof}

In the above proof, we also observe that

\begin{corollary} \label{primitive}
If $R$ is a prime right $\Sigma$-$V$ ring, then $R$ is a simple right Goldie ring. 
\end{corollary}

In view of the Proposition \ref{equivalent}, and Corollary \ref{primitive} it is useful to know when a simple right Goldie ring is a right $\Sigma$-$V$ ring.

\begin{proposition}
Let $R$ be a simple right Goldie ring. Then the following are equivalent:\\
{\rm (1)} $R$ is a right $\Sigma$-$V$ ring.\\
{\rm (2)} Every singular simple right $R$-module is $\Sigma$-injective. \\
{\rm (3)} $R$ is a right $V$-ring and q.f.d. relative to singular simple modules.
\end{proposition}
\begin{Proof}
$(1) \Rightarrow (2)$ is clear. \\
 $(2) \Rightarrow (1)$ follows from this fact that nonsingular 
quasi injective right modules over prime right Goldie ring are injective, see for example \cite [Theorem 8] {goodearlboyle}. \\
$(1) \Rightarrow (3)$ follows from Lemma \ref{qfd}.  \\
$(3) \Rightarrow (1)$. In the view of Lemma \ref{qfd}, it is enough to show that $R$ is q.f.d. relative to nonsingular
simple right $R$-modules. As a nonsingular simple module over a prime right Goldie ring is $\Sigma$-injective, a cyclic
module cannot contain an infinite direct sum of a copy of a nonsingular simple module. Thus it follows by Lemma \ref{qfd} that $R$ is a right $\Sigma$-$V$ ring. 
\end{Proof}

In \cite{Ashish} it is shown that for a von Neumann regular ring, the notion of $\Sigma$-$V$ ring is left-right symmetric. In the next theorem, we extend this left-right symmetry of $\Sigma$-$V$ rings to the class of exchange rings. Crawley and J\'{o}nnson introduced the notion of exchange property for modules in \cite{CJ}. A right $R$-module $M$ is said to satisfy the exchange property if for every right $R$-module $A$ and any two direct sum decompositions $A=M^{\prime}\oplus N=\oplus_{i\in \mathcal I}A_{i}$ with $M^{\prime} \simeq M$, there exist submodules $B_i$ of $A_i$ such that $A=M^{\prime} \oplus (\oplus_{i \in \mathcal I}B_i)$. If this hold only for $|\mathcal I|<\infty$, then $M$ is said to satisfy the finite exchange property. A ring $R$ is called an exchange ring if the module $R_R$ (or $_RR$) satisfies the (finite) exchange property. One of the equivalent characterizations of an exchange ring $R$ is that for each element $a\in R$, there exists an idempotent $e\in R$ such that $e\in aR$ and $1-e \in (1-a)R$. The class of exchange rings includes, in particular, von Neumann regular rings.  
 
\begin{theorem}
Let $R$ be an exchange right $\Sigma$-$V$ ring. Then 
\begin{enumerate}
\item $R$ is von Neumann regular.
\item $R$ has bounded index of nilpotence.
\end{enumerate}
\end{theorem}

\begin{proof}
\begin{enumerate}
\item Let $R$ be an exchange right $\Sigma$-$V$ ring. Let $R/I$ be a right primitive factor ring of $R$. Then clearly, $R/I$ is a right nonsingular right $\Sigma$-$V$ ring. It is known that a right nonsingular right $\Sigma$-$V$ ring has bounded index of nilpotence (see \cite [Corollary 6] {Ashish}), therefore $R/I$ has bounded index of nilpotent. Now, by \cite[Theorem 3]{yu}, $R/I$ is artinian. Since a ring all of whose right primitive factors are artinian is a right $\Sigma$-$V$ ring if and only if it is von Neumann regular \cite{Baccellaproc}, it follows that $R$ is von Neumann regular (see also \cite[Corollary 1.3]{FS2}).

\item Since von Neumann regular rings are right non-singular and right nonsingular right $\Sigma$-$V$ rings have bounded index of nilpotence \cite [Corollary 6] {Ashish}, we have that $R$ has bounded index of nilpotence.
\end{enumerate}
\end{proof}

\begin{remark}
{\rm  Part (2) of the above theorem gives a positive answer to Problem 15 of \cite{Ashish}.}
\end{remark}

In the view of the above theorem, we would like to ask

\bigskip

\noindent {\bf Question 2.} Is an exchange right $V$-ring necessarily von Neumann regular?

\bigskip

It seems highly unlikely that the answer to the above question would be ``yes" but to the best of our knowledge there are no known examples of exchange right $V$-ring that are not von Neumann regular.

\bigskip

\noindent As a consequence of the above theorem, we have

\begin{corollary}
An exchange ring is a right $\Sigma$-$V$ ring if and only if it is a left $\Sigma$-$V$ ring.
\end{corollary}

\begin{proof}
This follows from the fact that for a von Neumann regular ring, the property of being a $\Sigma$-$V$ ring is left-right symmetric. 
\end{proof}

Note that an exchange $V$-ring is not left-right symmetric. Example \ref{2.5} is an exchange right $V$-ring which is not a left $V$-ring.

\bigskip

\noindent The following characterizes prime right semi-hereditary right $\Sigma$-$V$ rings.
 
\begin{proposition} \label{primecase}
For a prime right $\Sigma$-$V$ ring $R$, the following are equivalent:\\
{\rm (1)} $R$ is right semi-hereditary.\\
{\rm (2)} $R$ is a matrix ring over a simple principal right ideal right $V$-domain.  
\end{proposition}

\begin{Proof}
 $(1)\Rightarrow(2)$. Let $R$ be a prime right semi-hereditary right $\Sigma$-$V$ ring. Then $E(R_R)$ is a semisimple ring and so it is $\Sigma$-injective as a right $R$-module. 
Thus by \cite [Proposition 12] {faith}, $R$ is right noetherian. On the other hand, $R$ is a simple ring with a 
 projective uniform right ideal and so by \cite{Hart} $R$ is Morita equivalent to a simple domain, say $D$. 
 Since $D$ is a right noetherian right hereditary domain, by \cite [Theorem 2.3] {Cohn}, $R$ is a principal right ideal domain. Now
 by \cite[Theorem 3 ] {Cohnfir}, it follows that $R$ is a matrix ring over a simple principal right ideal right $V$-domain.\\
$(2)\Leftarrow(1)$ is clear.
\end{Proof}

There exists an example of simple $\Sigma$-$V$-domain which is neither right nor left semi-hereditary. 
 
\begin{example}
{\rm By \cite [Theorem 5] {Boylehereditary}, if $R$ is a hereditary left noetherian ring, then $R$ is a right $QI$-ring if and only if it is a right 
noetherian right $V$-ring. 
Let $F$ be a universal differential field of characteristic zero with respect to two commuting derivations $\delta _1$ and $\delta _2$ (Kolchin \cite [Theorem, p. 771] {14}),
 and let $R = F[\theta _1, \theta _2]$ be the ring of linear differential
operators over $F$.  
We recall that the elements of $R$ are noncommutative
polynomials in the indeterminates $\theta _1, \theta _2$, subject to the relations
$ \theta _1 \theta _2 = \theta _2 \theta _1$ and $ \theta _ i \alpha = \alpha \theta _i + \delta _i \alpha $ for all $ \alpha \in  F$.
Then as it is shown in \cite  [Page 48 ]{goodearlboyle}, $R$ is a noetherian $V$-ring which is not a right $QI$ ring.
 Also, by \cite [Theorem 1] {CozzensJ}, $R$ is 
 a simple domain (and thus, a simple $\Sigma$-$V$ domain) which is not a right $QI$-ring and so it is not right hereditary. Note that a left semi-hereditary left noetherian ring 
 must be right semi-hereditary, see for example \cite{L. Small}. Therefore $R$ cannot be a left hereditary ring. }
\end{example}

\begin{proposition}
For a hereditary left noetherian ring $R$ the following are equivalent:\\
{\rm (1)} $R$ is a right $QI$-ring.\\
{\rm (2)} $R$ is a right noetherian right $V$-ring.\\
{\rm (3)} $R$ is a right $\Sigma$-$V$ ring.\\
{\rm (4)} $R$ is isomorphic to a finite direct sum of matrix ring over simple principal right ideal right $V$-domains.
\end{proposition}
\begin{Proof}
$(1) \Rightarrow (3)$ is clear and  $(1) \Leftrightarrow (2)$ is \cite [Theorem 5] {Boylehereditary}. \\
$(3) \Rightarrow (4)$. $R$ is semiprime left Goldie and so by \cite [Theorem 6.1] {Levy}, $R$ is
 an irredundant subdirect sum of a finite number of prime left Goldie rings $R_1, \cdots , R_n$. Thus there exist
  ideals $P_1, \cdots , P_n$ 
 of $R$ such that  $\cap _{i = 1} ^{n} P_i = 0$ and for each $1 \leq j \leq n$, $\cap _{i \neq j}  P_i \neq 0$ and for 
 each $i$, $R/P_i \cong R_i$. As each $R_i$ is a  prime right $\Sigma$-$V$ ring, it is a simple ring.
 Since  $\cap _{i \neq j}  P_i \nsubseteq P_j$, we have $\cap _{i \neq j}  P_i + P_j = R$. Thus $R$ is 
 isomorphic to $R_1 \times \cdots  \times R_n$. Each $R_i$ is a prime right hereditary right $\Sigma$-$V$ ring and so
  by Lemma \ref{primecase}, it is of the form of a matrix ring over a simple principal right ideal right $V$-domain.\\
  $(4) \Rightarrow (2)$ is straightforward. 
\end{Proof}

\bigskip

\noindent {{   \bf {Acknowledgments} }}. The authors would like to express their thanks to Professor Ken Goodearl for suggesting  Example \ref{2.5}.
The first and second authors would  like to thank  Iran National Science Foundation (INSF) and Yazd University for their support
 through the grant no. 94015014.

\bigskip

\bigskip

\end{document}